\documentclass[a4paper,11pt,openany,oneside,article]{memoir}

\usepackage[T1]{fontenc}
\usepackage[charter,cal=cmcal]{mathdesign}
\usepackage[scaled]{beramono,berasans}
\usepackage{textcomp}
\usepackage{microtype}
\usepackage{xspace}
\midsloppy

\usepackage{amsmath,bm}
\numberwithin{equation}{section}
\usepackage{amsthm}
\usepackage{thmtools}
  \declaretheorem[name=Theorem,within=section]{theorem}
  \declaretheorem[name=Lemma,sibling=theorem]{lemma}
  
  \declaretheorem[name=Definition,sibling=theorem,style=definition]{definition}

\usepackage{mathtools}
\usepackage{colonequals}
\usepackage{xcolor}
\definecolor{ugentblue}{RGB}{36,71,127}
\usepackage{tikz}
\usetikzlibrary{calc,arrows.meta,shapes.geometric}

\usepackage{hyperref}
  \hypersetup{hypertexnames=false,colorlinks,
              linkcolor=ugentblue,citecolor=ugentblue,urlcolor=ugentblue,
              bookmarksdepth=1, bookmarksopen=true,
              pdfstartview={XYZ null null 1}}
  
\usepackage[all]{hypcap}
\makeatletter
\let\blx@noerroretextools\@empty
\makeatother
\usepackage[backend=biber,style=alphabetic]{biblatex}
\usepackage[noabbrev,nameinlink,capitalise]{cleveref}
\let\etoolboxforlistloop\forlistloop
\usepackage{autonum}
\let\forlistloop\etoolboxforlistloop
\usepackage{csquotes}
\usepackage{multirow}
\usepackage[type={CC},modifier={by-nc-nd},version={4.0},imagemodifier={-eu}]{doclicense}

\newsubfloat{figure}
\usepackage{enumitem}
\firmlists
\usepackage{booktabs}
\captiondelim{. }
\captionnamefont{\bfseries\small}
\captiontitlefont{\small}
\changecaptionwidth
\captionwidth{0.95\linewidth}

\setsecnumdepth{subsection}

\newcommand\foreign[1]{#1}
\makeatletter
\newcommand\xperiod{\@ifnextchar.{}{.\@}}
\makeatother

\newcommand\ie{\foreign{i.e.}\xspace}

\newcommand\resp{\foreign{resp}\xperiod\xspace}
\newcommand\etc{\foreign{etc}\xperiod\xspace}

\newcommand*\dash{\nobreakdash-\hspace{0pt}}

\DeclarePairedDelimiter\angles{\langle}{\rangle}
\DeclarePairedDelimiterX\Set[2]{\lbrace}{\rbrace}{\,#1\,\delimsize\vert\,#2\,}

\newcommand{\1}{\bm{1}}

\definecolor{darkgreen}{RGB}{0,128,0}
\tikzset{
  edge/.style={ultra thick},
  edge0/.style={thin,red},
  edge1/.style={thick,darkgreen},
  vertex/.style={circle,fill,scale=0.5},
  vertex0/.style={vertex,red},
  vertex1/.style={vertex,darkgreen}
}

\title{Local Orientation-Preserving Symmetry Preserving Operations on Polyhedra}
\date{}
\author{
Pieter Goetschalckx \\
\footnotesize Ghent University \\
\footnotesize Krijgslaan 281-S9 \\
\footnotesize 9000 Ghent, Belgium \\
\footnotesize \url{pieter.goetschalckx@ugent.be}
\and
Kris Coolsaet \\
\footnotesize Ghent University \\
\footnotesize Krijgslaan 281-S9 \\
\footnotesize 9000 Ghent, Belgium \\
\footnotesize \url{kris.coolsaet@ugent.be}
\and
Nico Van Cleemput \\
\footnotesize Ghent University \\
\footnotesize Krijgslaan 281-S9 \\
\footnotesize 9000 Ghent, Belgium \\
\footnotesize \url{nico.vancleemput@gmail.com}
}
\hypersetup{pdftitle={Local Orientation-Preserving Symmetry Preserving Operations on Polyhedra},pdfauthor={Pieter Goetschalckx, Kris Coolsaet, Nico Van Cleemput}}

\begin{document}

\maketitle
\thispagestyle{empty}
\begin{tikzpicture}[remember picture,overlay]
\node[yshift=5em] at (current page.south) {\begin{minipage}{29.5em}\doclicenseThis\end{minipage}};
\end{tikzpicture}

\begin{abstract}
  \noindent Unifying approaches by amongst others Archimedes, Kepler, Goldberg, Caspar and Klug, Coxeter, and Conway, and extending on a previous formalisation of the concept of local symmetry preserving (lsp) operations, we introduce a formal definition of local operations on plane graphs that preserve orientation-preserving symmetries, but not necessarily orientation-reversing symmetries. This operations include, e.g., the chiral Goldberg and Conway operations as well as all lsp operations. We prove the soundness of our definition as well as introduce an invariant which can be used to systematically construct all such operations. We also show sufficient conditions for an operation to preserve the connectedness of the plane graph to which it is applied.
\end{abstract}

\section{Introduction}

Symmetry preserving operations on polyhedra have a long history -- from Plato and Archimedes to Kepler \cite{Kepler1619}, Goldberg \cite{Goldberg1937}, Caspar and Klug \cite{Caspar1962}, Coxeter \cite{Coxeter1971}, Conway \cite{Conway2008}, and many others. Notwithstanding their utility, until recently we had no unified way of defining or describing these operations without resorting to ad-hoc descriptions and drawings. In \cite{Brinkmann2017} the concept of local symmetry preserving operations on polyhedra (\emph{lsp operations} for short) was introduced. These replace each chamber in the barycentric sudbivision of a polyhedron with the same patch, which results in a new polyhedron while preserving the original symmetries. This established a general framework in which the class of all lsp operations can be studied, without having to consider individual operations separately. It was shown that many of the most frequently used operations on polyhedra fit into this framework.

However, some notable operations were not included. Most Goldberg operations and some of the extended Conway operations -- like \emph{snub} (see \cref{fig:snub}), gyro, propeller, \etc\ -- are chiral, so they only preserve orientation-preserving symmetries. In order to also cover these, we can generalize lsp operations by decorating double chambers instead of single chambers, similar to what Goldberg did in \cite{Goldberg1937} for Goldberg operations. We call these local orientation-preserving symmetry preserving (\emph{lopsp}) operations. In this paper, we formalize this approach for lopsp operations as \cite{Brinkmann2017} did for lsp operations.

\begin{figure}[htp]
  \centering
  \includegraphics[width=0.3\textwidth]{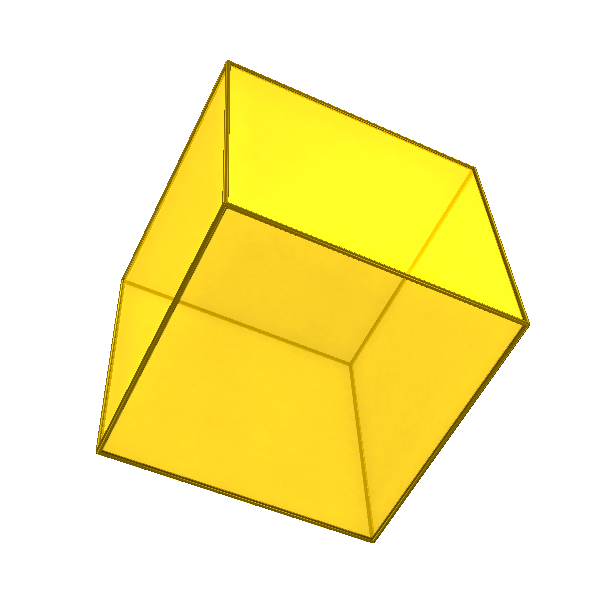}
  \raisebox{0.125\textwidth}{{\huge $\quad\rightarrow\quad$}}
  \includegraphics[width=0.3\textwidth]{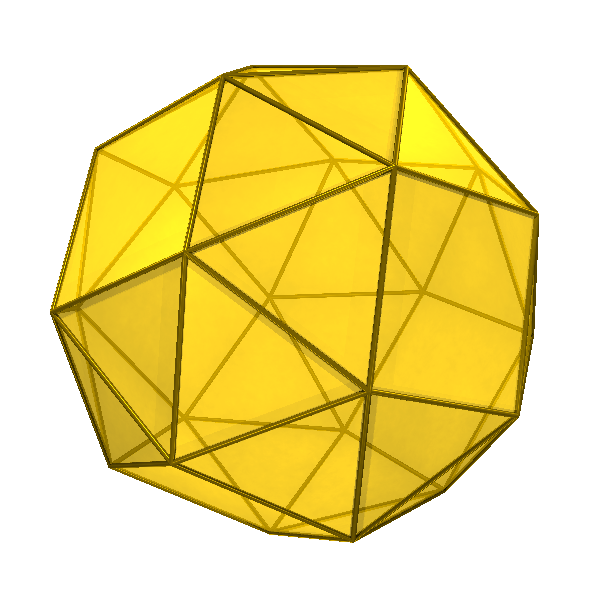}
  \caption{The Conway operation \emph{snub} applied to the cube, resulting in the Archimedean solid called \emph{snub cube}.}\label{fig:snub}
\end{figure}

In the remainder of this section we introduce a combinatorial characterization of plane graphs and the concept of \emph{chamber systems}. These allow us to define lopsp operations in \cref{sec:lopsp}. We prove that each lopsp operation can be represented by a \emph{double chamber patch}, but in contrast to the single chamber patch of a lsp operation, this one is not necessarily unique. We introduce the \emph{double chamber decoration} of an lopsp operation, which can be easily constructed from the double chamber patch but is independent of the chosen patch, and therefore unique for each lopsp operation. After some auxiliary results, we prove that the double chamber decoration is an invariant for equivalent lopsp operations. This makes it possible to identify a lopsp operation with its double chamber decoration. In \cref{sec:deco} we give a combinatorial characterization of double chamber decorations independent of the corresponding lopsp operation, and identify the double chamber decorations of \emph{2-connected} and \emph{3-connected} lopsp operations. Such a characterization is one of the first steps towards constructing a generation algorithm for lopsp operation as was done in \cite{decogen} for lsp operations. Finally, we prove that 2-connected \resp 3-connected lopsp operations preserve 2-connectivity \resp 3-connectivity, which makes it possible to see 3-connected lopsp operations as operations on polyhedra.

\subsection{Plane graphs and chamber systems}

In this paper, we will consider a plane graph as a rotation system on the set of directed edges.

\begin{definition}
  A \emph{plane graph} $G$ is a triple $(E, \sigma, \theta)$ where $E$ is a set of directed edges, $\sigma$ is a permutation of $E$ and $\theta$ is a fixed-point-free involution of $E$.
\end{definition}

This definition is equivalent to the more informal way of working with plane graphs.
The permutation $\sigma(e)$ gives the next edge with the same source vertex as $e$ in clockwise direction, and $\theta(e)$ gives the inverse edge of $e$. Note that the set of vertices is not explicitly defined, but can be retrieved as the set of orbits of $\angles{\sigma}$. The orbit corresponding to a vertex $v$ is the set of edges with source $v$. The faces correspond to orbits of $\angles{\sigma\theta}$, \ie the set of edges with the face to its left. The size of a face is the size of its corresponding orbit.

Every plane graph $G$ has an associated chamber system $C_G$ \cite{Dress1987}. This chamber system is obtained by constructing a barycentric subdivision of $G$, \ie subdividing each edge by one vertex in its center, adding one vertex in the center of each face, and adding edges from each center of a face to its vertices and centers of edges. In $C_G$, each vertex $v$ has a type $t(v) \in \{0, 1, 2\}$, indicating the dimension of its corresponding structure in $G$. Each edge $e$ has the type $t(e)$ of the opposite vertex in an adjacent triangles. A chamber system $C_G$ is a plane triangulation.

We call a pair of chambers sharing a type-$0$ edge a \emph{double chamber}. Each chamber of $C_G$ is contained in exactly one double chamber.

\begin{figure}[htp]
  \centering
  \begin{tikzpicture}[scale=0.5]
    \useasboundingbox (-5,-2.5) rectangle (11,8.5);

    \node[vertex0] (v1) at (0,6) {};
    \node[vertex0] (v2) at (6,6) {};
    \node[vertex0] (v3) at (0,0) {};
    \node[vertex0] (v4) at (6,0) {};
    \node[vertex0] (v5) at (2,4) {};
    \node[vertex0] (v6) at (4,4) {};
    \node[vertex0] (v7) at (2,2) {};
    \node[vertex0] (v8) at (4,2) {};

    \node[vertex1,scale=0.75] (e1) at (3,6) {};
    \node[vertex1,scale=0.75] (e2) at (6,3) {};
    \node[vertex1,scale=0.75] (e3) at (3,0) {};
    \node[vertex1,scale=0.75] (e4) at (0,3) {};
    \node[vertex1,scale=0.75] (e5) at (3,4) {};
    \node[vertex1,scale=0.75] (e6) at (4,3) {};
    \node[vertex1,scale=0.75] (e7) at (3,2) {};
    \node[vertex1,scale=0.75] (e8) at (2,3) {};
    \node[vertex1,scale=0.75] (e9) at (1,5) {};
    \node[vertex1,scale=0.75] (e10) at (5,5) {};
    \node[vertex1,scale=0.75] (e11) at (5,1) {};
    \node[vertex1,scale=0.75] (e12) at (1,1) {};

    \node[vertex,scale=0.75] (f0) at (3,8) {};
    \node[vertex,scale=0.75] (f1) at (3,5) {};
    \node[vertex,scale=0.75] (f2) at (5,3) {};
    \node[vertex,scale=0.75] (f3) at (3,1) {};
    \node[vertex,scale=0.75] (f4) at (1,3) {};
    \node[vertex,scale=0.75] (f5) at (3,3) {};

    \draw[edge0,thin] (f0) -- (e1) -- (f1) -- (e5) -- (f5);
    \draw[edge0,thin] (f0) .. controls (10,7) and (7,3) .. (e2) -- (f2) -- (e6) -- (f5);
    \draw[edge0,thin] (f0) .. controls (-4,7) and (-1,3) .. (e4) -- (f4) -- (e8) -- (f5);
    \draw[edge0,thin] (f0) .. controls (18,9) and (4,-8) .. (e3) -- (f3) -- (e7) -- (f5);
    \draw[edge0,thin] (f1) -- (e10) -- (f2) -- (e11) -- (f3) -- (e12) -- (f4) -- (e9) -- (f1);

    \draw[edge1] (f0) -- (v1) (f0) -- (v2)
                 (f0) .. controls (-4,8) and (-4,1) .. (v3)
                 (f0) .. controls (10,8) and (10,1) .. (v4);
    \draw[edge1] (f1) -- (v1) (f1) -- (v2) (f1) -- (v5) (f1) -- (v6);
    \draw[edge1] (f2) -- (v2) (f2) -- (v4) (f2) -- (v6) (f2) -- (v8);
    \draw[edge1] (f3) -- (v3) (f3) -- (v4) (f3) -- (v7) (f3) -- (v8);
    \draw[edge1] (f4) -- (v1) (f4) -- (v3) (f4) -- (v5) (f4) -- (v7);
    \draw[edge1] (f5) -- (v5) (f5) -- (v6) (f5) -- (v7) (f5) -- (v8);
    \draw[edge1] (f1) -- (v1) (f1) -- (v2) (f1) -- (v5) (f1) -- (v6);
    \draw[edge1] (f2) -- (v2) (f2) -- (v4) (f2) -- (v6) (f2) -- (v8);
    \draw[edge1] (f3) -- (v3) (f3) -- (v4) (f3) -- (v7) (f3) -- (v8);
    \draw[edge1] (f4) -- (v1) (f4) -- (v3) (f4) -- (v5) (f4) -- (v7);
    \draw[edge1] (f5) -- (v5) (f5) -- (v6) (f5) -- (v7) (f5) -- (v8);

    \draw[edge,ultra thick] (v1) -- (e1) -- (v2) -- (e2) -- (v4) -- (e3) -- (v3) -- (e4) -- (v1);
    \draw[edge,ultra thick] (v5) -- (e5) -- (v6) -- (e6) -- (v8) -- (e7) -- (v7) -- (e8) -- (v5);
    \draw[edge,ultra thick] (v1) -- (e9) -- (v5) (v2) -- (e10) -- (v6) (v4) -- (e11) -- (v8) (v3) -- (e12) -- (v7);
  \end{tikzpicture}
  \caption{The chamber system of the cube. Edges of type 0 are thin red, edges of type 1 are green and edges of type 2 are bold black. The green and black lines form the boundaries of the double chambers. The types of vertices can be derived from the incident edges.}\label{fig:barycentric}
\end{figure}
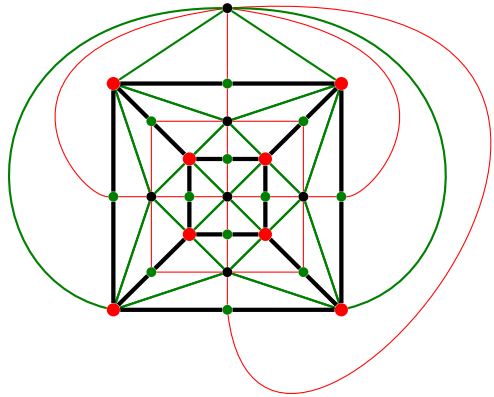

\subsection{Chiral operations}

An example of the construction of a chiral Goldberg operation is given in \cref{fig:goldberg}. A quadrangular \emph{double chamber patch} $v_1, v_0, v_2, v_0'$ consisting of the triangles $v_1, v_0, v_2$ and its counterpart $v_1, v_0', v_2$ is cut out of the hexagonal lattice $H$. Given a plane graph $G$ with chamber system $C_G$, we can glue this patch into each double chamber of $C_G$. The result is a plane graph $G'$ with the same orientation-preserving symmetries as $G$, but not necessarily the same orientation-reversing symmetries.

\begin{figure}[htp]
  \centering
  \begin{tikzpicture}[hexa/.style={shape=regular polygon,regular polygon sides=6,minimum size=1cm,draw,gray}]
    \foreach \j in {0,...,11} {
      \foreach \i in {0,...,4} {
        \node[hexa] (h\j;\i) at ({\i*3/2+3*mod(\j,2)/4},{\j*cos(30)/2}) {};
      }
    }
    \coordinate (x) at (0.5,0);
    \coordinate (y) at ({cos(60)/2},{sin(60)/2});
    \node[vertex] (v2) at ($2*(x)+2*(y)$) {};
    \node[vertex0] (v0) at ($(v2)+5*(x)+3*(y)$) {};
    \node[vertex0] (v0') at ($(v2)-3*(x)+8*(y)$) {};
    \node[vertex1] (v1) at ($0.5*(v0)+0.5*(v0')$) {};
    \node[below left] at (v2) {$v_2$};
    \node[right] at (v0) {$v_0$};
    \node[above] at (v0') {$v_0'$};
    \node[right] at (v1) {$v_1$};
    \draw[edge1] (v0) -- (v2) -- (v0');
    \draw[edge0] (v2) -- (v1);
    \draw[edge] (v0) -- (v1) -- (v0');
    \draw[very thick,dashed] (v2) -- ++($(x)+(y)$) -- ++($1.5*(x)$) -- ++($(y)-0.5*(x)$) -- ++($2*(x)-(y)$) -- ++($(x)+(y)$) -- (v0);
    \draw[very thick,dashed] (v2) -- +($2*(y)-(x)$) -- +($3.5*(y)-(x)$) -- +($4*(y)-2*(x)$) -- +($5*(y)-(x)$) -- +($7*(y)-2*(x)$) -- (v0');
    \draw[very thick,dashed] (v0) -- ++($-2*(x)$) -- ++($-2*(x)+2*(y)$) -- (v1) -- ++($0.5*(y)$) -- ++($-2*(x)+2*(y)$) -- (v0');
  \end{tikzpicture}
  \caption{The double chamber patch of a chiral Goldberg operation, with a simple path $P$ in dashed lines.}\label{fig:goldberg}
\end{figure}
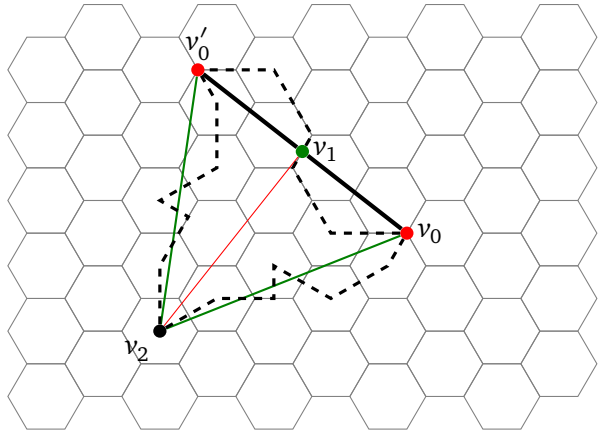

The symmetries of the hexagonal lattice ensure that after cutting and gluing the patch everything still fits together. But in order to have a combinatorial approach to the operations, we prefer to cut over a simple path $P$ in $C_H$ instead of cutting through edges and faces in arbitrary places. We will prove in \cref{lemma:simple-path} that it is always possible to find such a path.

It would be easy if we could split the double chamber patch into two separate triangles such that each triangle corresponds to the single chamber patch of an lsp operation \cite{Brinkmann2017}, and decorate each of the two types of chambers of $C_G$ with one of these two patches. Unfortunately, this is not always possible. In \cref{fig:snub-patch}, such an example is given. This is a double chamber patch for the lopsp operations \emph{snub}.

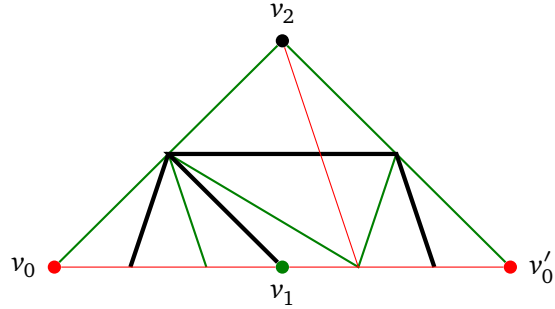
\begin{figure}[htp]
  \centering
  \begin{tikzpicture}
    \node[vertex1] (v1) at (0,0) {};
    \node[vertex] (v2) at (0,3) {};
    \node[vertex0] (v0) at (-3,0) {};
    \node[vertex0] (v0') at (3,0) {};

    \node[above=3pt] at (v2) {$v_2$};
    \node[left=3pt] at (v0) {$v_0$};
    \node[right=3pt] at (v0') {$v_0'$};
    \node[below=3pt] at (v1) {$v_1$};

    \draw[edge1] (v0) -- (v2) -- (v0') (-1,0) -- (-1.5,1.5) -- (1,0) -- (1.5,1.5);
    \draw[edge] (v1) -- (-1.5,1.5) -- (1.5,1.5) -- (2,0) (-1.5,1.5) -- (-2,0);
    \draw[edge0] (v0) -- (v1) -- (v0') (v2) -- (1,0);
  \end{tikzpicture}
  \caption{A double chamber patch for \emph{snub} (see \cref{fig:snub}).}\label{fig:snub-patch}
\end{figure}

\section{Lopsp operations}\label{sec:lopsp}

We define lopsp operations in a similar way to lsp operations, but instead of decorating each chamber with a single chamber patch, we will decorate double chambers.

\begin{definition}\label{def:lopsp}
Let $T$ be a connected tiling of the Euclidean plane with chamber system $C_T$, and let $v_0$ and $v_2$ be points in the Euclidean plane such that $v_0$ is the center of a rotation $\rho_{v_0}$ by 120 degrees in clockwise direction that is a symmetry of $T$ and $v_2$ is the center of a rotation $\rho_{v_2}$ by 60 degrees in clockwise direction that is a symmetry of $T$.

We call $(T, v_0, v_2)$ a \emph{local orientation-preserving symmetry preserving operation}, lopsp operation for short.
\end{definition}

Let $v_0' = \rho_{v_2}(v_0)$. The rotation $\rho_{v_1} = \rho_{v_2} \circ \rho_{v_0}$ is a rotation by 180 degrees with center $v_1$, and $v_0' = \rho_{v_1}(v_0)$.

In contrast to lsp operations, there is no obvious way to apply lopsp operations. We want to cut out the double chamber patch $v_2, v_0, v_1, v_0'$ and glue it into each double chamber, but the straight lines between these vertices do not always coincide with edges of $C_T$, and if we allow other cut-paths there are multiple possibilities (see \cref{fig:snub-alternative}).

\begin{figure}[htp]
  \centering
  \begin{tikzpicture}
    \node[vertex1] (v1) at (0,0) {};
    \node[vertex] (v2) at (0,3) {};
    \node[vertex0] (v0) at (-3,0) {};
    \node[vertex0] (v0') at (3,0) {};

    \node[above=3pt] at (v2) {$v_2$};
    \node[left=3pt] at (v0) {$v_0$};
    \node[right=3pt] at (v0') {$v_0'$};
    \node[below=3pt] at (v1) {$v_1$};

    \coordinate (l) at (-1,0);
    \coordinate (r) at (1,0);

    \draw[edge0] (v0) -- (v2) -- (v0') (v1) -- (1.5,1.5);
    \draw[edge1] (l) -- (v0) (l) -- (-1.5,1.5) (l) -- (v2) (l) -- (1.5,1.5) (r) -- (1.5,1.5) (r) -- (v0');
    \draw[edge] (l) -- (-2,1) (l) -- (-1,2) (l) -- (1,2) (l) -- (v1) -- (r) -- (2,1);
  \end{tikzpicture}
  \caption{Another double chamber patch for \emph{snub} (see \cref{fig:snub}).}\label{fig:snub-alternative}
\end{figure}
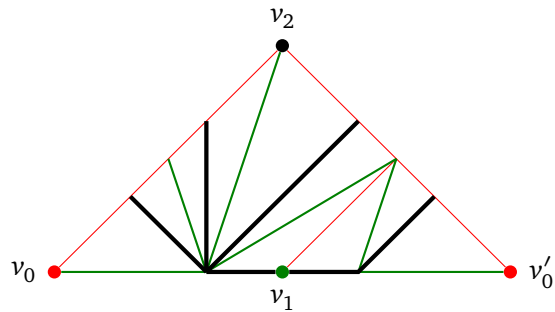

It is not difficult to imagine that no matter how we cut out this patch, the result after glueing them together will be the same. If we choose another path between $v_1$ and $v_0$ or between $v_0$ and $v_2$, we have to adapt the path between $v_1$ and $v_0'$ \resp $v_0'$ and $v_2$ accordingly, and the changes will cancel each other out when we glue the patches together.
This suggests that if we identify the vertices and edges on the border $v_1, v_0, v_2$ of the patch with the vertices and edges on the border $v_1, v_0', v_2$, the result is a triangulation of the sphere invariant under the chosen path. In \cref{fig:snub-decoration} the resulting triangulation for the \emph{snub} operation is given. We can even construct this triangulation without choosing a path.

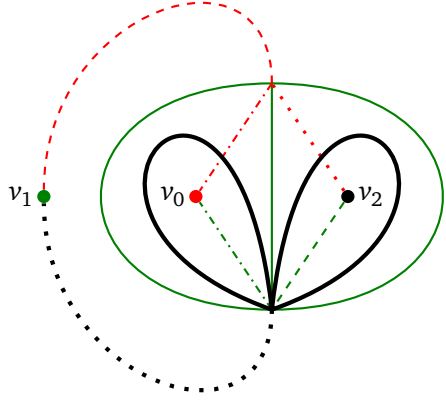
\begin{figure}[htp]
  \centering
  \begin{tikzpicture}[rotate=180]
    \useasboundingbox (-2.25,-1) rectangle (3,4);

    \node[vertex] (v2) at (-1,1.5) {};
    \node[vertex0] (v0) at (1,1.5) {};
    \node[vertex1] (v1) at (3,1.5) {};
    \coordinate (b) at (0,0);
    \coordinate (t) at (0,3);

    \node[right] at (v2) {$v_2$};
    \node[left] at (v0) {$v_0$};
    \node[left] at (v1) {$v_1$};

    \draw[edge1] (b) .. controls (-3,0) and (-3,3) .. (t);
    \draw[edge1] (b) .. controls (3,0) and (3,3) .. (t);
    \draw[edge1] (b) -- (t);
    \draw[edge1,dashed] (v2) -- (t);
    \draw[edge1,dash dot] (t) -- (v0);
    \draw[edge0,very thick,loosely dotted] (v2) -- (b);
    \draw[edge0,thick,dash dot] (b) -- (v0);
    \draw[edge0,thick,dashed] (b) .. controls (0,-2) and (3,-1) .. (v1);
    \draw[edge,loosely dotted] (t) .. controls (0,5) and (3,4) .. (v1);
    \draw[edge] (t) .. controls (-3.5,1.75) and (-0.5,-1.5) .. (t);
    \draw[edge] (t) .. controls (3.5,1.75) and (0.5,-1.5) .. (t);
  \end{tikzpicture}
  \caption{The double chamber decoration of \emph{snub} (see \cref{fig:snub}), with the simple path corresponding to the double chamber patches of \cref{fig:snub-patch,fig:snub-alternative} in dashed resp.\ dotted lines.}\label{fig:snub-decoration}
\end{figure}

For $C_T = (E, \sigma, \theta)$, consider the quotient set $\overline{E} = E / \angles{\rho_{v_0}, \rho_{v_2}}$. With $\overline{e}$ the equivalence class of $e$ in $\overline{E}$, we define
\begin{align}
  \overline{\sigma}(\overline{e}) &= \overline{\sigma(e)}, \\
  \overline{\theta}(\overline{e}) &= \overline{\theta(e)}, \\
  \overline{t}(\overline{e}) &= t(e).
\end{align}

\begin{definition}\label{def:decoration}
  The plane graph $(\overline{E}, \overline{\sigma}, \overline{\theta})$ described above together with labeling function $\overline{t}\colon \overline{E} \to \{0, 1, 2\}$ and special vertices $v_0$, $v_1$ and $v_2$ is called the \emph{double chamber decoration of the lopsp operation $(T, v_0, v_2)$}.
\end{definition}

Since for all $e_1, e_2 \in E$ with $\overline{e_1} = \overline{e_2}$ there exists a symmetry $\rho \in \angles{\rho_{v_0}, \rho_{v_2}}$ with $\rho(\sigma(e)) = \sigma(\rho(e))$, $\rho(\theta(e)) = \theta(\rho(e))$ and $\rho(t(e)) = t(\rho(e))$ such that $\rho(e_1) = e_2$, it is easy to prove that $\overline{\sigma}$, $\overline{\theta}$ and $\overline{t}$ are well-defined and $(\overline{E}, \overline{\sigma}, \overline{\theta})$ is indeed a plane graph.
Although we defined the labeling function $t$ only on edges, the types of vertices can be easily derived from the types of its incident edges.

\begin{lemma}\label{lemma:triangulation}
  The double chamber decoration of a lopsp operation is a plane triangulation.
\end{lemma}
\begin{proof}
  Since $C_T$ is a triangulation, we know that $(\sigma\theta)^3(e) = e$ for all edges $e \in E$. It follows immediately that $(\overline{\sigma}\overline{\theta})^3(\overline{e}) = \overline{e}$ for all edges $\overline{e} \in \overline{E}$.
  Since
  \begin{equation}
    \overline{t}(\overline{\sigma}\overline{\theta}(\overline{e})) = t(\sigma\theta(e)) \neq t(\theta(e)) = t(e) = \overline{t}(\overline{e}),
  \end{equation}
  it is impossible that $\overline{\sigma}\overline{\theta}(\overline{e}) = \overline{e}$ or $(\overline{\sigma}\overline{\theta})^2(\overline{e}) = \overline{e}$. Therefore, the size of each orbit of $\angles{\overline{\sigma}\overline{\theta}}$ is 3, which means that all the faces are triangles.
\end{proof}

Now that we obtained the double chamber decoration $D$ without choosing a cut-path, we can choose a path in $D$ instead of $C_T$. This is easier to do, because we do not have to take the symmetries into account. We can always find a path along the edges of $C_T$, without crossing through edges or faces.

\begin{lemma}\label{lemma:simple-path}
  If $D$ is the double chamber decoration of a lopsp operation, there exists a simple path $P$ between $v_1$ and $v_2$ through $v_0$.
\end{lemma}
\begin{proof}
  Since $D$ is a plane triangulation, it is 3-connected and therefore also 2-connected. It stays 2-connected if we temporarily add a vertex $w$ with edges to $v_1$ and $v_2$. By Menger's theorem \cite{Menger1927}, there exist two disjoint paths between $w$ and $v_0$. This is only possible if there are disjoint paths from $v_1$ to $v_0$ and from $v_0$ to $v_2$.
\end{proof}

We apply a double chamber decoration $D$ to a plane graph $G$ by cutting $D$ open along the simple path $P$ from the lemma above, which is the subdivided patch $v_1, v_0, v_2, v_0'$ that we glue into each double chamber of $G$. Instead of cutting and gluing, we can describe this application combinatorially.

Denote the set of directed edges on the path from $v_2$ to $v_0$ by $P_2$, and their inverses by $P_2'$. Denote the set of directed edges on the path from $v_1$ to $v_0$ by $P_1$, and their inverses by $P_1'$.

There is a one-to-one correspondence between the directed edges of a plane graph $G = (E, \sigma, \theta)$ and the double chambers of $C_G$, where each edge $e$ corresponds to the double chamber $c_e$ immediately to its left. The operations $s_1(c_e) = c_{\theta(e)}$ and $s_2(c_e) = c_{\sigma^{-1}\theta(e)}$ correspond to traversing the cyclic order around vertices of type 1 \resp 2. We call the set of double chambers of $G$ along with $s_1$ and $s_2$ the \emph{double chamber system} of $G$.

\begin{figure}[htp]
  \centering
  \begin{tikzpicture}[scale=2]
    \node[vertex] (v0) at (0, 0) {};
    \node[vertex] (v1) at (2, 0) {};
    \node[vertex] (v2) at (2, 2) {};

    \node[below left] at (v0) {$v_0$};
    \node[right=3pt] at (v1) {$v_1$};
    \node[above right] at (v2) {$v_2$};

    \draw[edge] (v1) -- (v0) -- (v2);
    \draw[thick,-Stealth] ($(v0)!0.25!(v2)+(v0)!1mm!90:(v2)$) -- node[above left] {$P_2'$} ($(v0)!0.75!(v2)+(v0)!1mm!90:(v2)$);
    \draw[thick,-Stealth] ($(v0)!0.75!(v2)+(v0)!-1mm!90:(v2)$) -- node[below right] {$P_2$} ($(v0)!0.25!(v2)+(v0)!-1mm!90:(v2)$);
    \draw[thick,-Stealth] ($(v0)!0.25!(v1)+(v0)!1mm!90:(v1)$) -- node[above] {$P_1'$} ($(v0)!0.75!(v1)+(v0)!1mm!90:(v1)$);
    \draw[thick,-Stealth] ($(v0)!0.75!(v1)+(v0)!-1mm!90:(v1)$) -- node[below] {$P_1$} ($(v0)!0.25!(v1)+(v0)!-1mm!90:(v1)$);
  \end{tikzpicture}
  \caption{A double chamber decoration with simple path $P$.}\label{fig:decoration}
\end{figure}
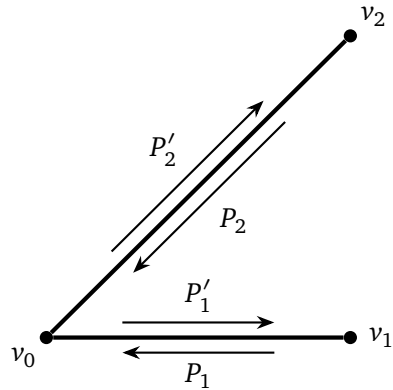

\begin{definition}\label{def:application}
  Given a plane graph $G$ with double chamber system $C$ and a double chamber decoration $D = (E, \sigma, \theta)$ with simple path $P$ satisfying \cref{lemma:simple-path}, the \emph{application of $(D, P)$ to $G$} results in a plane graph $D_P(G) = (E \times C, \sigma_P, \theta_P)$ with $\sigma_P((e, c)) = (\sigma(e), s_{P,e}(c))$ and $\theta_P((e, c)) = (\theta(e), s_{P,e}(c))$ where
  \begin{equation}
    s_{P,e} = \begin{cases}
      s_2 & \text{if } e \in P_2 \\
      s_1^{-1} & \text{if } e \in P_1 \\
      s_2^{-1} & \text{if } e \in P_2' \\
      s_1 & \text{if } e \in P_1' \\
      \1 & \text{else} \\
    \end{cases}
  \end{equation}
\end{definition}

It is possible that there is more than one simple path that satisfies \cref{lemma:simple-path}. We still have to prove that the result of the operation does not depend on the chosen path $P$. We will do that in \cref{thm:equivalence}, but we first introduce some new terminology.

Given a double chamber decoration $D$ with two simple paths $P$ and $Q$ satisfying \cref{lemma:simple-path}, consider the subgraph of $C_D$ consisting of all the edges in $P$ and $Q$ (see \cref{fig:snub-regions} for an example). In order to avoid confusion, we will refer to the faces of this subgraph as \emph{regions}. With each directed edge $e$ of $C_D$ we associate exactly one region $R_e$. If $e$ is an edge in $P$ or $Q$ we choose the region at the left-hand side of $e$, and for all other edges we choose the containing region.
A \emph{region path} $R_0, \dotsc, R_n$ is a sequence of regions such that for each $i < n$ there exists an edge $e_i \in Q \setminus P$ such that $e_i$ is associated with $R_i$ and $\theta(e_i)$ is associated with $R_{i + 1}$. A region path corresponds to the operation $r_1 \circ \dotsb \circ r_n$ with $r_i = s_{Q,e_i}$. Two region paths are called equivalent if they correspond to the same operation.

\begin{figure}[htp]
  \centering
  \begin{tikzpicture}[scale=2]
    \useasboundingbox (-0.4, -0.7) rectangle (2, 2);
    \node[vertex0] (v0) at (0, 0) {};
    \node[vertex1] (v1) at (2, 0) {};
    \node[vertex] (v2) at (2, 2) {};
    \node[below left] at (v0) {$v_0$};
    \node[right=3pt] at (v1) {$v_1$};
    \node[above right] at (v2) {$v_2$};
    \coordinate (01) at (1,0);
    \coordinate (02) at (1,1);

    \draw[edge,thick,dash dot] (01) -- (v0) -- (02);
    \draw[edge,thick,dashed] (v1) -- (01) (02) -- (v2);
    \draw[edge,very thick,loosely dotted] (02) .. controls (1.5,1) and (2,0.5) .. (v1);
    \draw[edge,very thick,loosely dotted] (01) .. controls (0,-2) and (-2,1) .. (v2);
  \end{tikzpicture}
  \caption{The regions of the double chamber decoration of \emph{snub} (see \cref{fig:snub-decoration}) with two simple paths.}\label{fig:snub-regions}
\end{figure}

\begin{lemma}\label{lemma:region}
  Given a double chamber decoration $D$ with two simple paths $P$ and $Q$ satisfying \cref{lemma:simple-path}, there exists a region $R_{P,Q}$ such that there is a region path between $R_{P,Q}$ and a region incident to $v_2$ with an associated operation of the form $s_2^k$, and a region path between $R_{P,Q}$ and a region incident to $v_1$ with an associated operation of the form $s_1^l$.
\end{lemma}
\begin{proof}
  Choose a region path $R = R_0, \dotsc, R_n$ with $R_0$ incident to $v_1$ and $R_n$ incident to $v_2$ and associated operation $r = r_1 \dotsb r_n$. Such a region path exists because $P$ contains no cycles, so all regions are connected.
  If we add one vertex in each region and one vertex on each edge in $Q$, and an edge between a vertex in a region and a vertex on an edge if the edge is in the border of the region, this region path induces a path on these edges in a canonical way, and each operation $r_i$ corresponds to an intersection of $R$ and $Q$. An example is given in \cref{fig:region-path}.

  \begin{figure}[htp]
    \centering
    \subbottom[]{
      \label{fig:region-path}
      \begin{tikzpicture}[scale=2]
        \useasboundingbox (-0.25, -0.75) rectangle (2.5, 2.25);
        \node[vertex] (v0) at (0, 0) {};
        \node[below left] at (v0) {$v_0$};
        \node[vertex] (v1) at (2, 0) {};
        \node[right=3pt] at (v1) {$v_1$};
        \node[vertex] (v2) at (2, 2) {};
        \node[above right] at (v2) {$v_2$};

        \draw[edge,thick] (v1) -- (v0) -- (v2) node[pos=0.75,above left] {$P$};
        \draw[edge,thick,dashed] (v0) .. controls ($(v0)!0.3!75:(v2)$) and ($(v0)!0.4!45:(v2)$) .. (1, 1) node[pos=0.75,above] {$Q$} .. controls ($(v2)!0.4!45:(v0)$) and ($(v2)!0.3!75:(v0)$) .. (v2);
        \draw[edge,thick,dashed] (v0) .. controls ($(v0)!2.6!30:(v1)$) and (1.5, -2.5) .. (1, 0) .. controls (0.9, 0.5) and (2, 0.6) .. (2, 0);
        \draw[edge0,thick] (1.5, 0.2) -- (2.1, 0.5) -- (2.2, 1) node[right] {$R$} -- (1.75, 1.4);
      \end{tikzpicture}
    }
    \hspace{1em}
    \subbottom[]{
      \begin{tikzpicture}
        \useasboundingbox (-2.5,-1.5) rectangle (1,5);
        \draw[edge0,thick] (0,-0.5) -- (0,0.5) (0,4) -- (0,4.5) node[above] {$R$};
        \draw[edge0,thick,dashed] (0,0.5) -- (0,4);

        \draw[edge,thick] (0.25,0) node[black,right] {$r_i$} -- (0,0) arc[start angle=270, end angle=90, radius=2] (0,4) node[left,pos=0.5] {$Q_t$} -- (0.25,4) node[black,right] {$r_j$};

        \draw[edge,thick] (0.25,0.5) node[right] {$r_a$} -- (0,0.5) arc[start angle=270, end angle=90, radius=1] (0,2.5) node[left,pos=0.5] {$Q_s$} -- (0.25,2.5) node[right] {$r_b$};
      \end{tikzpicture}
      \label{fig:region}
    }
    \caption{Examples of region paths}
  \end{figure}
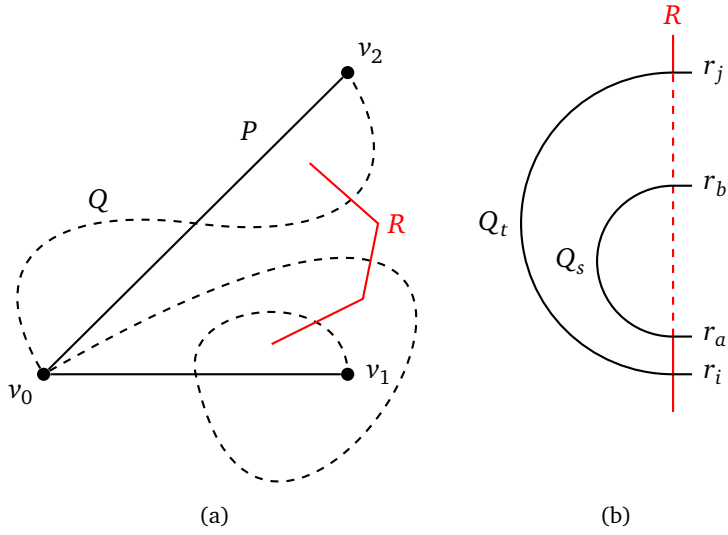

  We will prove that if $r_i$ and $r_j$ correspond to two intersections of $R$ and $Q_t$ with $t \in \{1, 2\}$ that are consecutive on $Q_t$, then $r_{i + 1} \dotsb r_{j - 1}$ is the identity operation.
  For $j = i + 1$ this is obvious. Suppose $j > i + 1$.
  The subpath of $Q_t$ between $r_i$ and $r_j$ together with the region path between $r_i$ and $r_j$ forms a closed cycle. If there is an intersection with $Q_s$ in $r_a$ with $a = i + 1$, there will be another intersection in $r_b$ with $a < b < j$ and $r_b = r_a^{-1}$, as illustrated in \cref{fig:region}.
  We can assume by induction that $r_a \dotsb r_b = r_a r_b = \1$. If $b < j - 1$, we can repeat this for $a = b + 1$ until $b = i + 1$ and thus $r_{i + 1} \dotsb r_{j - 1} = \1$.

  Take $m$ so that $r_m \in \{s_1, s_1^{-1}\}$ and $r_i \in \{s_2, s_2^{-1}\}$ for all $i > m$. If $Q_2$ crosses $R$ in $r_a$ with $1 < a < m$ and $r_i \in \{s_1, s_1^{-1}\}$ for $i < a$, it will cross again in $r_b$ with $a < b < m$, and $r_a \dotsb r_b = \1$.
  We can repeat this as long as there is an $r_c \in \{s_2, s_2^{-1}\}$ with $b < c < m$, until $r_1 \dotsb r_m = s_1^k$. Since $r_{m + 1} \dotsb r_n = s_2^l$, the region between $r_m$ and $r_{m + 1}$ satisfies the conditions of $R_{P,Q}$.
\end{proof}

We are now ready to prove that the application of a double chamber decoration $D = (E, \sigma, \theta)$ to a graph $G$ with double chamber system $C$ is independent of the chosen simple path $P$. In order to do that, we will construct an isomorphism between the plane graphs $D_P(G)$ and $D_Q(G)$, with $Q$ another simple path satisfying \cref{lemma:simple-path}. By choosing the region $R_{P,Q}$, we fix canonical points $R_{P,Q} \times C$ that will be invariant under this isomorphism.

\begin{theorem}\label{thm:equivalence}
  Given a plane graph $G$ with double chamber system $C$ and a double chamber decoration $D = (E, \sigma, \theta)$ with two simple paths $P$ and $Q$ satisfying \cref{lemma:simple-path}, there exists an isomorphism between $D_P(G)$ and $D_Q(G)$.
\end{theorem}
\begin{proof}
  Choose a region $R_{P,Q}$ satisfying \cref{lemma:region} and consider the function
  \begin{equation}
    f\colon E \times C \to E \times C : (e, c) \mapsto (e, s_{P,Q,e}(c))
  \end{equation}
  with $s_{P,Q,e}$ the operation associated with a region path from $R_{P,Q}$ to $R_e$.
  We will first prove that $f$ is a homomorphism between $D_P(G)$ and $D_Q(G)$. Since
  \begin{align}
    \sigma(f((e, c))) &= \sigma((e, s_{P,Q,e}(c))) = (\sigma(e), s_{Q,e}s_{P,Q,e}(c)) \\
    f(\sigma((e, c))) &= f((\sigma(e), s_{P,e}(c))) = (\sigma(e), s_{P,Q,\sigma(e)}s_{P,e}(c)), \\[1em]
    \theta(f((e, c))) &= \theta((e, s_{P,Q,e}(c))) = (\theta(e), s_{Q,e}s_{P,Q,e}(c)) \\
    f(\theta((e, c))) &= f((\theta(e), s_{P,e}(c))) = (\theta(e), s_{P,Q,\theta(e)}s_{P,e}(c)),
  \end{align}
  we only have to prove that $s_{Q,e} s_{P,Q,e} = s_{P,Q,\sigma(e)} s_{P,e}$.

  Consider the case $s_{P,e} = \1$. The operation $s_{P,Q,\sigma(e)}$, corresponding to a region path from $R_{P,Q}$ to $R_{\sigma(e)}$, is equal to $s_{P,Q,e}$ followed by the operation corresponding to the region path crossing $e$, which is $s_{Q,e}$. Therefore, $s_{P,Q,\sigma(e)} s_{P,e} = s_{Q,e} s_{P,Q,e}$.

  If $s_{P,e} = s_1$, there is a region path from $R_{P,Q}$ to $R_e$ consisting of a region path from $R_{P,Q}$ to a region $R_1$ incident to $v_1$, corresponding to operation $s_1^k$, followed by a region path from $R_1$ to $R_e$, corresponding to operation $r$. Since $e \in P_1'$, the region path from $R_1$ to $R_e$ can follow the left-hand side of $P_1$.
  The region path from $R_{P,Q}$ to $\sigma(e)$ starts with the same region path to $R_1$. We can now follow the region path along the right-hand side of $P_1$ to $R_{\sigma(e)}$, corresponding to operation $r'$, after we go around $v_1$ which corresponds to operation $s_1^{-1}$. In \cref{fig:equivalence}, we see that $r'$ is equal to $r$ followed by $s_{Q,e}$. Therefore,
  \begin{equation}
    s_{P,Q,\sigma(e)} s_{P,e} = r' s_1^{-1} s_1^k s_1 = s_{Q,e} r s_1^k = s_{Q,e} s_{P,Q,e}.
  \end{equation}

  For $s_{P,e}$ equal to $s_1^{-1}$, $s_2$ and $s_2^{-1}$, the proof is similar.

  \begin{figure}[htp]
    \centering
    \begin{tikzpicture}[scale=1.5]
      \useasboundingbox (0,-0.5) rectangle (2.5,2);
      \node[vertex] (v0) at (0, 0) {};
      \node[vertex] (v1) at (2, 0) {};
      \node[vertex] (v2) at (2, 2) {};

      \draw[edge,thick,dashed] (v1) -- (v0) -- (v2);
      \draw[edge1,ultra thick,-stealth] (0.25,0) -- node[above] {$e$} (0.5,0);
      \draw[edge,thick] (0.6,-0.25) -- (0.75,0) -- (1,0) -- (1.15,0.25);
      \draw[edge,thick] (1.1,-0.25) -- (1.4,0.25);
      \draw[edge,thick] (1.35,-0.25) -- (1.5,0) -- (1.75,0) -- (1.9,-0.25);
      \draw[edge,thick] (2.5,0) -- (v1);

      \node[circle,fill=red,scale=0.25] (r) at (2,1) {};
      \node[circle,fill=red,scale=0.25] (r1) at (2,0.15) {};
      \node[above right,red] at (r) {$R_{P,Q}$};
      \draw[edge0,thick,dotted] (r) -- node[right] {$s_1^k$} (r1);
      \draw[edge0,thick,dotted,-{Latex[right]}] (r1) -- node[above=5] {$r$} (0.4,0.15);
      \draw[edge0,thick,dotted,-{Latex[left]}] (r1) arc[radius=0.15,start angle=90,end angle=-90] (2,-0.15) -- node[below=5] {$r'$} (0.4,-0.15);
    \end{tikzpicture}
    \hspace{1em}
    \begin{tikzpicture}[scale=1.5]
      \useasboundingbox (0,-0.5) rectangle (2.5,2);
      \node[vertex] (v0) at (0, 0) {};
      \node[vertex] (v1) at (2, 0) {};
      \node[vertex] (v2) at (2, 2) {};

      \draw[edge,thick,dashed] (v1) -- (v0) -- (v2);
      \draw[edge,thick] (0.65,0.25) -- (0.5,0) -- (1.25,0) -- (1.4,-0.25);
      \draw[edge1,ultra thick,-stealth] (0.75,0) -- node[above] {$e$} (1,0);
      \draw[edge,thick] (1.75,0.25) -- (1.75,0) -- (v1);

      \node[circle,fill=red,scale=0.25] (r) at (2,1) {};
      \node[circle,fill=red,scale=0.25] (r1) at (2,0.15) {};
      \node[above right,red] at (r) {$R_{P,Q}$};
      \draw[edge0,thick,dotted] (r) -- node[right] {$s_1^k$} (r1);
      \draw[edge0,thick,dotted,-{Latex[right]}] (r1) -- node[above=5] {$r$} (0.9,0.15);
      \draw[edge0,thick,dotted,-{Latex[left]}] (r1) arc[radius=0.15,start angle=90,end angle=-90] (2,-0.15) -- node[below=5] {$r'$} (0.9,-0.15);
    \end{tikzpicture}
    \hspace{1em}
    \begin{tikzpicture}[scale=1.5]
      \useasboundingbox (0,-0.5) rectangle (2.5,2);
      \node[vertex] (v0) at (0, 0) {};
      \node[vertex] (v1) at (2, 0) {};
      \node[vertex] (v2) at (2, 2) {};

      \draw[edge,thick,dashed] (v1) -- (v0) -- (v2);
      \draw[edge,thick] (0.65,0.25) -- (0.5,0) -- (1.25,0) -- (1.4,0.25);
      \draw[edge1,ultra thick,-stealth] (0.75,0) -- node[above] {$e$} (1,0);
      \draw[edge,thick] (1.75,-0.25) -- (1.75,0) -- (v1);

      \node[circle,fill=red,scale=0.25] (r) at (2,1) {};
      \node[circle,fill=red,scale=0.25] (r1) at (2,0.15) {};
      \node[above right,red] at (r) {$R_{P,Q}$};
      \draw[edge0,thick,dotted] (r) -- node[right] {$s_1^k$} (r1);
      \draw[edge0,thick,dotted,-{Latex[right]}] (r1) -- node[above=5] {$r$} (0.9,0.15);
      \draw[edge0,thick,dotted,-{Latex[left]}] (r1) arc[radius=0.15,start angle=90,end angle=-90] (2,-0.15) -- node[below=5] {$r'$} (0.9,-0.15);
    \end{tikzpicture}
    \caption{Some examples for $s_{P,e} = s_1$}\label{fig:equivalence}
  \end{figure}
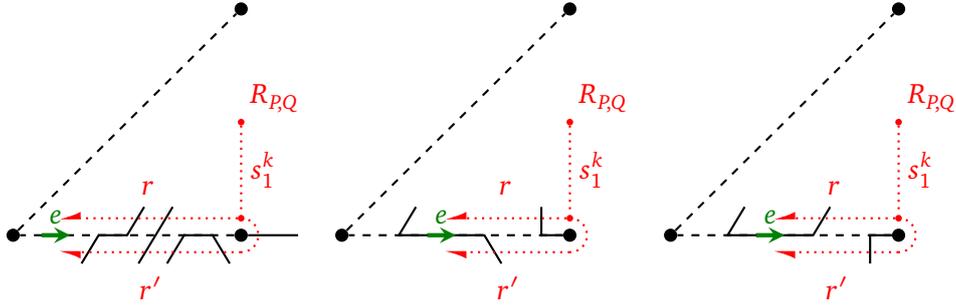

  Suppose $f((e, c)) = f((e', c'))$, \ie $(e, s_{P,Q,e}(c)) = (e', s_{P,Q,e'}(c'))$. It follows immediately that $e = e'$, and since $s_{P,Q,e}$ is a permutation we have that $c = c'$. Thus $(e, c) = (e', c')$ and $f$ is injective. For each $(e, c) \in E \times C$, $f((e, s_{P,Q,e}^{-1}(c))) = (e, c)$, and thus $f$ is surjective.

  Since $f$ is a bijective homomorphism, it is an isomorphism between $D_P(G)$ and $D_Q(G)$.
\end{proof}

\section{Double chamber decorations}\label{sec:deco}

In the previous section we constructed the double chamber decoration for a given lopsp operation. This double chamber decoration contains all the necessary information in order to apply the decoration to an embedded graph, but does not depend on the tiling $T$ or the simple path $P$ chosen to define and apply the lopsp operation. Since two lopsp operations are equivalent if and only if they have the same double chamber decoration, it is easier to work with the double chamber decorations directly instead of deriving them from lopsp operations. But in order to do that, we need a full characterization of these graphs. This is similar to what we did for lsp operations in \cite{decogen}.

\begin{theorem}
  A plane triangulation $D$ with vertex set $V$ and edge set $E$, together with a labeling function $t\colon V \cup E \to \{0,1,2\}$ and three special vertices $v_0, v_1, v_2$ is a double chamber decoration of a lopsp operation if and only if
  \begin{enumerate}
    \item for each edge $e = (v, w)$, $\{t(e), t(v), t(w)\} = \{0, 1, 2\}$
    \item for each vertex~$v$ with~$t(v) = i$, the types of the edges incident to $v$ are alternating between~$j$ and~$k$ with~$\{i, j, k\} = \{0, 1, 2\}$
    \item for each vertex~$v$ different from $v_0, v_1, v_2$
      \begin{align}
        t(v) = 1 \quad&\Rightarrow\quad \operatorname{deg}(v) = 4 \\
        t(v) \neq 1 \quad&\Rightarrow\quad \operatorname{deg}(v) > 4
      \end{align}
      and
      \begin{gather}
        t(v_0), t(v_2) \neq 1 \\
        \operatorname{deg}(v_0), \operatorname{deg}(v_2) \geq 2 \\
        t(v_1) = 1 \quad\Rightarrow\quad \operatorname{deg}(v_1) = 2 \\
        t(v_1) \neq 1 \quad\Rightarrow\quad \operatorname{deg}(v_1) \geq 4
      \end{gather}
  \end{enumerate}
\end{theorem}
\begin{proof}
  It is easy to verify that the double chamber decoration of a lopsp operation satisfies these properties.

  Given a graph $D$ that satisfies the properties, there exists a simple path $P$ between $v_1$ and $v_2$ through $v_0$, since the proof of \cref{lemma:simple-path} holds for all plane triangulations. We can cut $D$ open along this path to get a subdivided patch $D'$, and glue this patch into each double chamber of the hexagonal lattice $H$. The result will be a chamber system $C_T$ of a tiling $T$.

  We will now prove that the type-2 subgraph of $D'$, consisting of all type-2 edges, is connected. Let $u$ and $v$ be two vertices in the type-2 subgraph. Since every face of $D'$ is a cycle, $D'$ is 2-connected. Menger's theorem \cite{Menger1927} gives us that there exist two vertex-disjoint paths between $u$ and $v$. Since all faces of $D'$ except for the outer face are triangles, these two paths form a cycle with only triangles on the inside. Since $u$ is in the type-2 subgraph, it has type 0 or 1, and there is an edge $(u, u')$ of type 2 on or in the cycle.
  If $u' \neq v$, we can do the same for vertices $u'$ and $v$, and we can choose a cycle that contains less triangles than the previous one. By induction, there exists a path between $u$ and $v$ in the type-2 subgraph of $D'$.

  Given vertices $u$ and $v$ in the type-2 subgraph of $C_T$, there exists a sequence of chambers $C_0, \dotsc, C_n$ of $H$ such that two consecutive chambers $C_i$ and $C_{i + 1}$ share one side, and $u$ is contained in $C_0$ and $v$ in $C_n$. Since there are at least two vertices on each side of $D'$, and they are not both of type 2, at least one of them is in the type-2 subgraph of $C_T$. Thus, there is a type-2 path between $u$ and $v$ that passes through all chambers in the sequence $C_0, \dotsc, C_n$, and the type-2 subgraph of $C_T$ is connected. It follows immediately that $T$ is connected too.

  We can choose the vertices of one double chamber of $C_H$ in $T$ as $v_0$, $v_1$, $v_0'$ and $v_2$. Now $(T, v_0, v_2)$ satisfies \cref{def:lopsp} of a lopsp operation, and the double chamber decoration of this lopsp operation is $D$.
\end{proof}

We call a lopsp operation and the corresponding double chamber decoration $k$-connected if it is derived from a $k$-connected tiling $T$. For the following results, we need a lemma from \cite{decogen}, which we will repeat here without proof.

\begin{lemma}\label{lemma:connectivity}
  A plane graph $G$ is
  \begin{enumerate}
    \item 2\dash connected if and only if $C_G$ contains no type-1 cycles of length 2.
    \item 3\dash connected if and only if $G$ is 2\dash connected and $C_G$ contains no non-empty type-1 cycles of length 4.
  \end{enumerate}
\end{lemma}

\begin{theorem}
  If $G$ is a $k$-connected plane graph with $k \in \{1, 2, 3\}$ and $O$ is a $k$-connected lopsp operation, then $O(G)$ is a $k$-connected plane graph.
\end{theorem}
\begin{proof}
  Suppose $O$ is derived from a tiling $T$ as in \cref{def:lopsp}.

  For $k = 1$, we know that $T$ and $G$ are connected, and it follows easily that $O(G)$ is connected too.

  For $k = 2$, we will prove that $O(G)$ is 2-connected.
  A type-1 cycle of length 2 in $C_{O(G)}$ is either completely contained in an area that was one double chamber of $C_G$ before it was subdivided by $O$, or it is split between two areas of adjacent double chambers. Both cases cannot appear, as for any double chamber (\resp any pair of adjacent double chambers) of $C_G$ there is an isomorphism between the area of this double chamber (\resp two double chambers) in $C_{O(G)}$ and the corresponding area in $T$, and $T$ is connected and thus has no type-1 cycles of length 2.
  This implies that $C_{O(G)}$ contains no type-1 cycles of length 2.

  For $k = 3$, we will prove that $O(G)$ is 3-connected.
  Consider a type-1 cycle $C$ with edges $e_1, \dotsc, e_n$ in $C_{O(G)}$. Let $C_1, \dotsc, C_n$ be a sequence of double chambers of $C_G$ such that $e_i$ is contained in the area of $C_i$ for $1 \leq i \leq n$. If $C_i = C_{i + 1}$, we can remove $C_i$ from the sequence. This results in the reduced sequence $C_1, \dotsc, C_m$.

  If $C$ has length 2, then $m \leq 2$. Thus, the cycle is contained in one or two neighbouring areas, and it should be present in the tiling $T$ too, which is impossible.

  If $C$ has length 4, then $m \leq 4$. Thus, the cycle is contained in the areas of at most 4 double chambers of $C_G$, and each double chamber has at least one vertex or edge in common with the previous and next one, but not the same for both of them. We will now construct a type-1 cycle in $C_G$ though these chambers. Depending on the position of the common elements in each double chamber, we choose type-1 edges of $C_G$ as in \cref{fig:doubleconnected}.

  \begin{figure}[htp]
    \centering
    \loosesubcaptions
    \begin{tabular}[c]{cccc}
      \subbottom[]{
        \begin{tikzpicture}
          \draw[edge,thick] (-1,0) -- (1,0);
          \draw[edge0,solid,line width=2.5,line cap=round] (-1,0) -- (0,1) -- (1,0);
        \end{tikzpicture}
        \label{fig:doubleconnected1}
      } &
      \subbottom[]{
        \begin{tikzpicture}
          \draw[edge1,dashed] (-1,0) -- (0,1);
          \draw[edge0,solid,line width=2.5,line cap=round] (0,1) -- (1,0) -- (-1,0);
        \end{tikzpicture}
        \label{fig:doubleconnected2}
      } &
      \subbottom[]{
        \begin{tikzpicture}
          \draw[edge,thick] (-1,0) -- (1,0);
          \draw[edge1,dashed] (-1,0) -- (0,1);
          \draw[edge0,solid,line width=2.5,line cap=round] (0,1) -- (1,0);
          \node[vertex0,circle,scale=1.25] (v) at (-1,0) {};
        \end{tikzpicture}
        \label{fig:doubleconnected3}
      } &
      \subbottom[]{
        \begin{tikzpicture}
          \draw[edge,thick] (-1,0) -- (1,0) -- (0,1);
          \node[vertex0,circle,scale=1.25] (v) at (-1,0) {};
          \draw[edge0,line width=2.5,line cap=round] (-1,0) -- (0,1);
        \end{tikzpicture}
        \label{fig:doubleconnected4}
      } \\[0.5cm]
      \subbottom[]{
        \begin{tikzpicture}
          \draw[edge,thick] (-1,0) -- (1,0) -- (0,1);
          \node[vertex0,circle,scale=1.25] (v) at (0,1) {};
          \draw[edge0,solid,line width=2.5,line cap=round] (-1,0) -- (0,1);
        \end{tikzpicture}
        \label{fig:doubleconnected5}
      } &
      \subbottom[]{
        \begin{tikzpicture}
          \draw[edge,thick] (0,1) -- (1,0);
          \draw[edge1,dashed] (-1,0) -- (0,1);
          \node[vertex0,circle,scale=1.25] (v) at (0,1) {};
          \draw[edge0,solid,line width=2.5,line cap=round] (-1,0) -- (1,0);
        \end{tikzpicture}
        \label{fig:doubleconnected6}
      } &
      \subbottom[]{
        \begin{tikzpicture}
          \draw[edge,thick] (-1,0) -- (0,1);
          \draw[edge1,dashed] (1,0) -- (0,1);
          \node[vertex0,circle,scale=1.25] (v) at (0,1) {};
          \draw[edge0,solid,line width=2.5,line cap=round] (-1,0) -- (1,0);
        \end{tikzpicture}
        \label{fig:doubleconnected6b}
      } &
      \subbottom[]{
        \begin{tikzpicture}
          \draw[edge,thick] (0,1) -- (1,0) -- (-1,0);
          \draw[edge1,dashed] (-1,0) -- (0,1);
          \node[vertex0,circle,scale=1.25] (u) at (-1,0) {};
          \node[vertex0,circle,scale=1.25] (v) at (0,1) {};
        \end{tikzpicture}
        \label{fig:doubleconnected7}
      } \\[0.5cm]
      \subbottom[]{
        \begin{tikzpicture}
          \draw[edge,thick] (-1,0) -- (1,0);
          \draw[edge1,dashed] (-1,0) -- (0,1) -- (1,0);
          \node[vertex0,circle,scale=1.25] (u) at (-1,0) {};
          \node[vertex0,circle,scale=1.25] (v) at (1,0) {};
        \end{tikzpicture}
        \label{fig:doubleconnected8}
      } &
      \subbottom[]{
        \begin{tikzpicture}
          \draw[edge,thick] (-1,0) -- (0,1) -- (1,0);
          \node[vertex0,circle,scale=1.25] (u) at (-1,0) {};
          \draw[edge0,solid,line width=2.5,line cap=round] (-1,0) -- (1,0);
        \end{tikzpicture}
        \label{fig:doubleconnected9}
      } &
      \subbottom[]{
        \begin{tikzpicture}
          \draw[edge1,dashed] (-1,0) -- (0,1) -- (1,0);
          \node[vertex0,circle,scale=1.25] (u) at (-1,0) {};
          \draw[edge0,solid,line width=2.5,line cap=round] (-1,0) -- (1,0);
        \end{tikzpicture}
        \label{fig:doubleconnected9b}
      }
    \end{tabular}
    \caption{The thick red vertices and edges are the ones in common with the previous and next double chambers. The dashed green edges are part of the chosen cycle. The choice between \subcaptionref{fig:doubleconnected6} or \subcaptionref{fig:doubleconnected6b} and \subcaptionref{fig:doubleconnected9} or \subcaptionref{fig:doubleconnected9b} depends on the choice in the double chamber below (the cycle has to be connected).}\label{fig:doubleconnected}
  \end{figure}
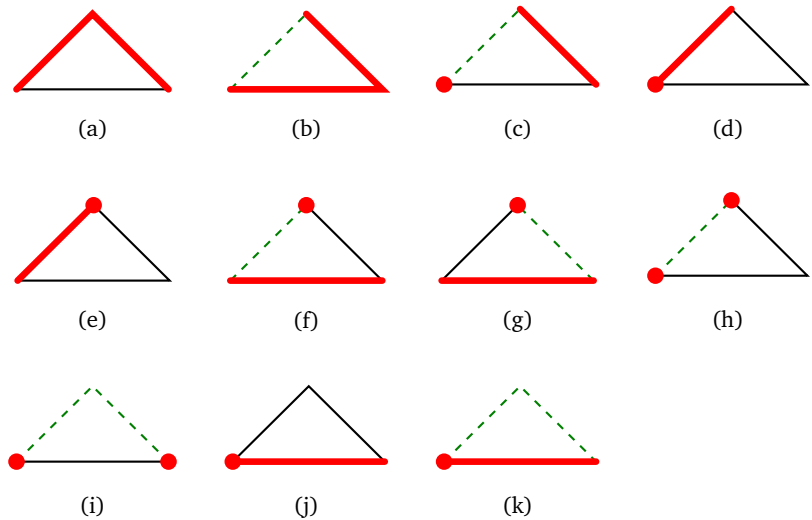

  This results in at most 4 edges that form a type-1 cycle or single edge $C'$ in $C_G$. In \cref{fig:doubleconnected8,fig:doubleconnected9b} we choose two edges, but since $v_0$ and $v_0'$ are of the same type, the path between them on the cycle $C$ has to be at least of length 2 too. If $C'$ is a type-1 cycle, it has to be empty since $G$ is 3-connected. Thus, the situation is as in \cref{fig:type1cycle1}.

  The type-1 cycle $C$ in $C_{O(G)}$ is completely contained in the areas of double chambers of $C_G$ adjacent to $C'$. The only situation where this would not necessarily imply a type-1 cycle in $C_T$ is when $C$ is a cycle of length 4 surrounding a type-2 vertex. This implies that $C$ passes though 3 or 4 areas corresponding to double chambers of $C_G$, as illustrated in \cref{fig:type1cycle3,fig:type1cycle4}. There are at least two areas that contain only one edge of $C$. But since all the areas are isomorphic, it is easy to see that this is impossible.

  \begin{figure}[htp]
    \centering
    \subbottom[]{
      \begin{tikzpicture}
        \draw[edge,thick,dashed] (0,0) -- (0,2) -- (-2,2) -- (-1,1) -- (-2,0) -- (2,0) -- (1,1) -- (2,2) -- (0,2);
        \draw[edge,thick,dashed] (0,2) -- (-0.5,2.5) (0,2) -- (0,2.5) (0,2) -- (0.5,2.5);
        \draw[edge,thick,dashed] (0,0) -- (-0.5,-0.5) (0,0) -- (0,-0.5) (0,0) -- (0.5,-0.5);
        \draw[edge,thick,dashed] (-1,1) -- (-1.5,1.125) (-1,1) -- (-1.5,0.875);
        \draw[edge,thick,dashed] (1,1) -- (1.5,1.125) (1,1) -- (1.5,0.875);
        \draw[edge1,very thick] (0,0) -- (-1,1) -- (0,2) -- (1,1) -- cycle;
      \end{tikzpicture}
      \label{fig:type1cycle1}
    }
    \hspace{2em}
    \subbottom[]{
      \begin{tikzpicture}[scale={4/3}]
        \useasboundingbox ({-sqrt(3)/2}, -0.875) rectangle ({sqrt(3)/2}, 1.375);
        \draw[edge,thick,dashed] (90:1) -- (-30:1) -- (-150:1) -- (90:1) (90:1) -- (0,0) -- (-150:1) (0,0) -- (-30:1);
        \draw[edge1,very thick] (-150:0.75) -- (-30:0.4);
        \draw[edge,thick] (-30:0.75) -- (90:0.4) (90:0.75) -- (-150:0.4);
      \end{tikzpicture}
      \label{fig:type1cycle3}
    }
    \hspace{2em}
    \subbottom[]{
      \begin{tikzpicture}
        \useasboundingbox (-1, -1.5) rectangle (1, 1.5);
        \draw[edge,thick,dashed] (-1,-1) -- (-1,1) -- (1,1) -- (1,-1) -- (-1,-1) -- (1,1) (-1,1) -- (1,-1);
        \draw[edge1,very thick] (-0.75,-0.75) -- (0.4,-0.4);
        \draw[edge,thick] (0.75,-0.75) -- (0.4,0.4) (0.75,0.75) -- (-0.4,0.4) (-0.75,0.75) -- (-0.4,-0.4);
      \end{tikzpicture}
      \label{fig:type1cycle4}
    }
    \caption{The type-1 cycle $C'$ with adjacent double chambers.}\label{fig:type1cycle}
  \end{figure}
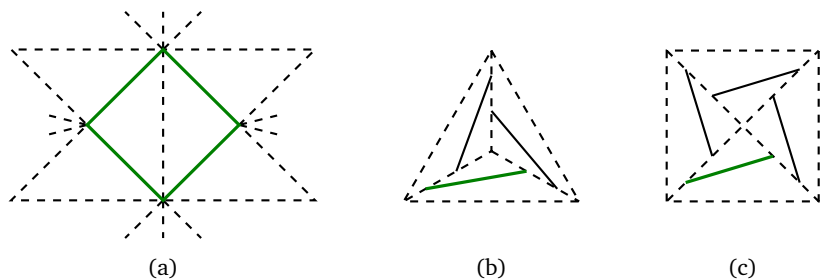
\end{proof}

This theorem is particularly interesting for $k = 3$, for which it says that 3\dash connected lopsp operations are operations on polyhedra.

\section*{Acknowledgements}

The authors wish to thank the anonymous referee whose suggestions greatly improved the presentation of this work.

\printbibliography

\end{document}